\documentclass[11pt,a4paper]{article}

\usepackage[T1]{fontenc} 
\usepackage[utf8]{inputenc} 
\usepackage[frenchb,english]{babel} 
\usepackage{fourier}
\usepackage[scaled=0.875]{helvet}
\usepackage{courier}

\usepackage{amsmath,amsthm,amssymb,amsfonts,epic,latexsym,hyperref,natbib
}
\usepackage[dvips]{graphicx}
\usepackage[usenames]{color}

\newtheorem{theo}{Theorem}[section]
\newtheorem{defi}[theo]{Definition}
\newtheorem{prop}[theo]{Proposition}

\newtheorem{assu}{Assumption}

\newtheorem{exam}{Example}

\newtheorem{lemm}[theo]{Lemma}
\newtheorem{rema}[theo]{Remark}

\def\R{\mathbb{R}}
\def \N{\mathbb{N}}
\def \Z{\mathbb{Z}}

\def\1{\mathbf{1}}
\def \PP{\mathbf{P}} 
\def\EE{\mathbf {E}} 
\def \E{\mathbb{E}} 
\def \Es{\mathbb{E}^\star}
\def \Et{\mathbb{E}^\theta}
\def \Pt{\mathbb{P}^\theta}
\def \P{\mathbb{P}} 
\def \Ps{\mathbb{P}^\star}

\newcommand{\ts}{{\theta^{\star}}}
\newcommand{\htet}{\widehat \theta_n}
\newcommand{\PPt}{\mathbf{P}^{\theta}}

\newcommand{\PPs}{\mathbf{P}^{\star}}

\newcommand{\EEs}{\mathbf{E}^{\star}}

\def\x{\mathbf {x}} 
\def\h{\mathbf{h}} 
\def \eps{\varepsilon}
\def\w{\omega}
\def\t{\theta}

\def\to{\rightarrow}

\newcommand{\argmax}{\mathop{\rm Argmax}}

\definecolor{pink}{rgb}{1,0.3,0.8}

\def\dd{\mathrm{d}}
\def\ee{\mathrm{e}}
\def\to{\rightarrow}

\def\Beta{\mathrm{B}}

\begin{document}

\DeclareGraphicsExtensions{.pdf, .jpg, .jpeg, .png, .ps}

\title{
Maximum likelihood estimator consistency for ballistic random walk in a parametric random environment.}

\author{Francis {\sc Comets}\footnote{
Laboratoire Probabilit\'es et Mod\`eles Al\'eatoires, Universit\'e
  Paris     Diderot,      UMR     CNRS     7599,      E-mail:     {\tt
    comets@math.univ-paris-diderot.fr};
$^{\dag}$ Laboratoire Statistique et G\'enome, Universit\'e d'\'Evry Val
d'Essonne, UMR CNRS 8071, USC INRA,     E-mail:                          {\tt
  $\{$mikael.falconnet, catherine.matias$\}$@genopole.cnrs.fr}; 
$^\ddag$ D\'epartement Informatique, IUT de Fontainebleau,
E-mail: {\tt oleg.loukianov@u-pec.fr};
$^{\S}$ Laboratoire Analyse et
Probabilit\'es, Universit\'e d'\'Evry Val d'Essonne, 
E-mail: {\tt dasha.loukianova@univ-evry.fr}
}
 \and 
 Mikael {\sc Falconnet}$^{\dag}$ 
 \and 
Oleg {\sc Loukianov}$^{\ddag}$ 
\and 
Dasha {\sc Loukianova}$^{\S}$ 
\and 
Catherine {\sc Matias}$^{\dag}$
}

\maketitle

\begin {abstract}
We consider  a one  dimensional ballistic random  walk evolving  in an
i.i.d. parametric random environment.  We provide a maximum likelihood
estimation procedure  of the parameters based on  a single observation
of the  path till the time it  reaches a distant site,  and prove that
the estimator is consistent as  the distant site tends to infinity. We
also explore the numerical performances of our estimation procedure. 
\end{abstract}


\section{Introduction}

Random walks in random environments (RWRE) have attracted much attention
lately, mostly in the physics and probability theory literature.  
These processes were  introduced originally by~\cite{Chernov} to model
the replication of a DNA sequence. The idea underlying Chernov's model
is that the protein that 
moves along the  DNA strand during replication performs  a random walk
whose transition 
probabilities depend on the sequence letters, thus modeled as a random
environment. Since then, RWRE have been developed far beyond this
original  motivation, resulting  into a  wealth of  fine probabilistic
results.  Some recent surveys on the subject 
include~\cite{Hughes} and~\cite{ZeitouniSF}.

 Recently,  these models have regained interest  from biophysics, as
 they fit the description of  some physical experiments that unzip the
 double strand of a DNA  molecule.  More precisely, some fifteen years
 ago,  the first  experiments on  unzipping a  DNA sequence  have been
 conducted, relying on several different techniques \citep[see][and
 the  references  therein]{Balda_06,Balda_07}.   By that  time,  these
 experiments 
 primarily took place in  the quest for alternative (cheaper,faster or
 both) sequencing 
 methods.  When conducted in the presence of bounding proteins, such 
 experiments also enabled the identification of specific locations at which
 proteins and enzymes bind to the DNA~\citep{Koch_02}.  
Nowadays, similar  experiments  are conducted  in order  to
 investigate  molecular  free  energy  landscapes  with  unprecedented
 accuracy \citep{Alemany,Huguet_Forns_Ritort}.  Among other
 biophysical applications, one  can mention the study  of the formation  of DNA or
 RNA hairpins \citep{Bizarro}.

Despite the  emergence of data that  is naturally modeled  by RWRE, it
appears that very few statistical issues on those processes 
have been studied so far. Very recently, \cite{andreo} considered a 
problem      inspired      by      an      experiment      on      DNA
unzipping \citep{Balda_06,Balda_07,Cocco_Monasson}, where the aim is to predict
the sequence of bases relying on the observation of several unzipping 
of  one finite  length DNA  sequence. Up  to some  approximations, the
problem  boils   down  to  considering   independent  and  identically
distributed (i.i.d.) replicates of a one dimensional nearest neighbour 
path (\emph{i.e.}  the walk has $\pm 1$ increments) in the same 
finite and two-sites  dependent environment, up to the  time each path
reaches some value $M$ (the sequence 
length).  In this setup, the authors consider both a discrete time and a
continuous time model.  They  provide  estimates  of  the  values  of  the
environment at each site, which corresponds to estimating the sequence
letters of  the DNA molecule. Moreover, they  obtain explicit formula
for the probability to be wrong for a given estimator, thus evaluating
the quality of the prediction.

In the present work, we study a different problem, also 
motivated by  some DNA unzipping experiments: relying  on an arbitrary
long trajectory of 
a transient  one-dimensional nearest neighbour path, we  would like to
estimate the parameters of the 
environment's distribution.  Our  motivation comes more precisely from
the most recent experiments, that aim at characterising free binding energies
between base pairs relying on the unzipping of a synthetic DNA
sequence \citep{Ribezzi}.  In this setup, the environment is still
considered as random as those free energies are unknown and need to be
estimated. While our asymptotic  setup is still far from corresponding
to the reality of those experiments, our work might give some insights on
statistical properties of estimates of those binding free energies. 


The parametric estimation of the environment distribution 
has already  been studied in~\cite{AdEn}.  In their  work, the authors
consider  a  very  general  RWRE  and provide  equations  relying  the
distribution of some statistics of 
the trajectory to some moments of the environment distribution. 
In the specific case of a one-dimensional nearest neighbour path, 
those equations give moment estimators for the environment distribution parameters.  
It is worth
mentioning that due to its great generality, the method is hard to
understand at  first, but it takes  a simpler form  when one considers
the specific case of a one-dimensional nearest-neighbour path. 
Now, the method has two main drawbacks: first, it is not generic in the
sense that it has to be designed differently for each parametric setup
that  is considered.  Namely, the  method relies  on the  choice  of a
one-to-one mapping between the  parameters and some moments. Note that
the  injectivity  of  such a  mapping  might  even  not be  simple  to
establish      (see      for       instance      the      case      of
Example~\ref{ex:deuxpts_3param}    below,    further   developed    in
Section~\ref{sect:EstiAdEn}). 
Second,  from a  statistical  point of  view,  it is  clear that  some
mappings  will give  better  results than  others.  Thus the  specific
choice of a mapping has an impact on the estimator's performances.

As an alternative, we propose here to consider maximum likelihood
estimation  of  the parameters  of  the  environment distribution.  We
consider a  transient nearest neighbour path in  a random environment,
for which we are able to define some criterion - that we call a
log-likelihood  of the observed  process, see~\eqref{eq:l}  below. Our
estimator is then defined as the maximiser of this criterion - thus a
maximum likelihood  estimator.  When properly  normalised, we prove
that this criterion is convergent as the size of the path increases
to infinity.  This part of our  work relies on using  the link between
RWRE and branching processes in random environments (BPRE). While this
link is already  well-known in the literature, we  provide an explicit
characterisation  of  the  limiting  distribution  of  the  BPRE  that
corresponds to our RWRE (see theorem~\ref{theo:limitlaw} below). 
Relying on this precise characterisation, we then further
prove that the limit of our
normalised  criterion is  finite in  what is  called  the ballistic
region, namely the  set of parameters such that the  path has a linear
increase (see  Section~\ref{sect:RWRE} below for  more details). Then,
following standard  statistical results, we are able  to establish the
consistency of our estimator.   We also provide  synthetic experiments to
compare    the    effective     performances    of    our    estimator
and~\citeauthor{AdEn}'s       procedure.       In      the       cases
where~\citeauthor{AdEn}'s estimator is easily settled, while the two methods
exhibit  the  same  performances  with  respect  to  their  bias,  our
estimator exhibits a much smaller variance. 
We mention that establishing asymptotic normality of
this estimator  requires much  more technicalities and  is out  of the
scope of the  present work. This point will be  studied in a companion
article, together with variance estimates and confidence intervals. 

The   article  is   organised   as  follows.   Section~\ref{sect:RWRE}
introduces our  setup: the one dimensional nearest  neighbour path, and
recalls some well-known results about the behaviour of those 
processes.  Then   in  Section~\ref{sect:M_Estimator},  we   present  the
construction  of  our  M-estimator   (\emph{i.e.}  an  estimator
maximising some criterion function), state the assumptions required on the
model as well as our consistency result (Section~\ref{sect:consistency_statement}). Section~\ref{sect:ex} presents some
examples of  environment distributions for which  the model assumptions
are satisfied so that our  estimator is consistent.  Now, the proof of
our        consistency       result       is        presented       in
Section~\ref{sect:consistency}.  The section  starts by  recalling the
link between  RWRE and  BPRE (Section~\ref{sect:BP}).  Then,  we state
our  core  result:  the  explicit  characterisation  of  the  limiting
distribution of the branching process that is linked with our path; and
its corollary: the existence of a (possibly infinite) limit for the normalised
criterion (Section~\ref{sect:LLN}). 
In Section~\ref{sect:UnifConv} we first provide a technical result on the
uniformity of  this convergence, then establish that  in the ballistic
case, the limit of the normalised criterion is finite. An almost converse
statement is also given  (Lemma~\ref{lemm:ue}). To conclude this part,
we  prove  in  Section~\ref{sect:Ident}  that the  limiting  criterion
identifies the  true parameter value (under  a natural identifiability
assumption  on the model  parameter).  Finally,  numerical experiments
are  presented  in Section~\ref{sect:experiments},  focusing  on the  three
examples that were developed in Section~\ref{sect:ex}.  Note that we also provide an explicit
description  of  the  form  of~\citeauthor{AdEn}'s  estimator  in  the
particular  case  of the  one-dimensional  nearest  neighbour path  in
Section~\ref{sect:EstiAdEn}.


\section{Definitions, assumptions and results}


\subsection{Random walk in random environment} \label{sect:RWRE}
Let  $\w=\{\w_x\}_{x\in\Z}$  be  an  independent  and  identically
distributed  (i.i.d.) collection  of  $(0,1)$-valued random  variables
with  distribution  $\nu$.  The  process~$\w$  represents a  random
environment in which the random walk will evolve.  We suppose that the
law    $\nu=\nu_{\theta}$     depends    on    some  unknown  parameter
$\theta\in\Theta,$ where $\Theta \subset \R^d$ is assumed to be a compact set.
Denote  by  $\P^{\theta}=\nu_{\theta}^{\otimes   \Z}$  the  law  on
$(0,1)^{   \Z}$  of   the   environment  $\{\w_x\}_{x\in\Z}$   and by 
$\E^{\theta}$ the expectation under this law.

For fixed  environment $\w$, let $X=\{X_t\}_{t\in\N}$  be the Markov
chain on $\Z$ starting at $X_0=0$ and with transition probabilities 
\[
  P_{\w}(X_{t+1}=y|X_t=x)=\left \{\begin{array}{lr}
\w_x&\mbox{if}\ y=x+1,\\
1-\w_x&\mbox{if}\ y=x-1,\\
0&\mbox{otherwise}.
\end{array}\right.
\]
The symbol $P_{\w}$ denotes the measure on the path space of $X$ given
$\w$, usually called \emph{quenched} law. The (unconditional) law of $X$ is
given by 
\[
  \PP^{\theta}(\cdot)=\int P_{\omega}(\cdot)\dd\P^{\theta}(\omega),
\]
this is the so-called
\emph{annealed}  law.   We  write  $E_{\w}$  and   $\EE^{\theta}$  for  the
corresponding quenched and annealed expectations, respectively. We start to  recall some 
well-known asymptotic results. 
Introduce a family of i.i.d. random variables,
\begin{equation}
  \label{eq:rho}
  \rho_x=\frac{1-\w_x}{\w_x}, \qquad x \in {\mathbb Z},
\end{equation}
and assume that $\log \rho_0$ is integrable.
\cite{Sol} proved the following 
classification:
\begin{itemize}
\item[(a)] 
if $\E^{\theta}(\log\rho_0)<0$, then 
\[
  \displaystyle{  \lim_{t   \to  \infty}   X_t  =  +   \infty},  \quad
  \PP^{\theta} \mbox{-almost surely}. 
\] 
\item[(b)] 
If $\E^{\theta}(\log\rho_0)=0$, then 
\[
  \displaystyle{   - \infty = \liminf_{t \to \infty} X_t < \limsup_{t \to \infty} X_t =
  +\infty}, \quad  \PP^{\theta} \mbox{-almost surely}.
\]
\end{itemize}
The case of $ \E^{\theta}(\log\rho_0)>0 $ follows from (a) by changing
the sign of $X$. Note that the walk $X$ is 
$\PP^{\theta}$-almost surely   transient in case (a) and recurrent in case (b). 

 In  the present  paper,  we restrict  to  the case  (a)  when $X$  is
 transient to the right. 
Then,  it was also found that
the rate of its increase (with respect to time $t$) is either linear or slower
than linear. The first case  is called \emph{ballistic} case and the
second one \emph{sub-ballistic}  case. More precisely, letting $T_n$
be the first hitting time of the positive integer $n$, 
\begin{equation} \label{equa:HittingTime}
T_n = \inf \{ t \in \N \, : \, X_t = n \},
\end{equation}
and assuming $\E^{\theta}(\log\rho_0)<0$ all through, we have
\begin{itemize}
\item[(a1)] if  $\E^{\theta}(\rho_0) < 1$,   then,  $  \PP^{\theta} \mbox{-almost surely},$
  \begin{equation}
\label{eq:lln1}
\frac{T_n}{n}           \xrightarrow[n          \to          \infty]{}
\frac{1+\E^{\theta}(\rho_0)}{1-\E^{\theta}(\rho_0)} 
,
\end{equation}
\item[(a2)]  If   $\E^{\theta}(\rho_0)  \geq 1$,  then $T_n   /   n \to + \infty$  
$ \PP^{\theta} \mbox{-almost surely}$, when $n$ tends to infinity.
  \end{itemize}
 



\subsection{Construction  of a M-estimator} \label{sect:M_Estimator}
  
We  address the  following statistical  problem: estimate  the unknown
parameter $\theta$ 
from a single observation of the  RWRE path till the time it reaches a
distant site.  Assuming 
transience to  the right, we  then observe $X_{[0,T_n]}=\{X_t \,  : \,
t=0,1,\ldots, T_n\}$, for some $n \geq 1$.

If  $\x_{[0,t]}:=(x_0,\dots,x_t)$ is a nearest neighbour path of length $t$, we define
 for all $x\in\Z,$ 
 \begin{align}
&L(x,\x_{[0,t]}):=\sum_{s=0}^{t-1}\1\{x_s=x;\
x_{s+1}=x-1\}, \label{equa:LeftStep} \\
\mbox{and} \quad 
&R(x,\x_{[0,t]}):=\sum_{s=0}^{t-1}\1\{x_s=x;\ x_{s+1}=x+1\}, \label{equa:RightStep}
\end{align}
the number  of left  steps (resp.  right steps)    from site~$x$.  (Here,  $\1\{\cdot\}$ denotes  the
indicator function). We let also $v_t$ (resp. $V_{T_n}$) be the set of
integers visited by the path $\x_{[0,t]}$ (resp. $X_{[0,T_n]}$).
Consider now  a nearest neighbour path $\x_{[0,t_n]}$  starting from 0
and first hitting 
site $n$ at time $t_n$.  It is straightforward to compute its quenched
and annealed probabilities, respectively
\[ 
  P_{\w}(X_{[0,T_n]}=\x_{[0,t_n]})=\prod_{x\in
  v_{t_n}}\w^{R(x,\x_{[0,t_n]})}_x (1-\w_x)_{{\phantom{x}}}^{L(x,\x_{[0,t_n]})}
\]
and
\[
\PP^{\theta}(X_{[0,t_n]}=\x_{[0,t_n]}) = \prod_{x\in
  v_{t_n}}\int_0^1a^{R(x,\x_{[0,t_n]}) }(1-a)^{L(x,\x_{[0,t_n]})}\dd \nu_{\theta}(a).
\]
Under the following assumption, these weights add up to 1 over all possible choices 
of $\x_{[0,t_n]}$.
\begin{assu}(Transience to the right). \label{assu:trans}
 For any $ \theta\in\Theta$,
$  \E^{\theta}|\log\rho_0|<\infty$ and
\[
  \E^{\theta}(\log\rho_0)<0.
\]
\end{assu}

Introducing the short-hand notation
\[
  L_x^n := L(x,X_{[0,T_n]}) \quad \mbox{and} \quad R_x^n := R(x,X_{[0,T_n]}),
\]
we can express the  (annealed) log-likelihood of the observations as
\begin{equation} \label{eq:vrsbl2}
\tilde \ell_n(\theta) =
\sum_{x=0}^{n-1}\log\int_0^1a^{R_{x}^n}(1-a)^{L_x^n}\dd\nu_{\theta}(a)  
 +      \sum_{x<0,x\in      V_{T_n}}\log
 \int_0^1a^{R_x^n}(1-a)^{L_x^n}\dd\nu_{\theta}(a). 
\end{equation}
Note that as the random walk $X$ starts from
$0$  (namely $X_0=0$)  and is  observed until  the first  hitting time
$T_n$ of $n \geq 1$, 
we  have $R_x^n=L_{x+1}^n+1$ for  
$x=1,2,\ldots,n-1$. We will  perform this change in the  first line of
the right-hand side of 
(\ref{eq:vrsbl2}).
Also, since   the   walk  is transient   to   the   right
(Assumption~\ref{assu:trans}),  the second sum  in the  right-hand side
(accounting for negative sites  $x$) is almost surely bounded.  Hence,
this sum will not influence in a significant way the behaviour of the
normalised log-likelihood, and we will  drop it. Therefore, we are led
to the following choice. 
\begin{defi}
Let   $\phi_{\theta}$  be the  function  from $\N^2$ to $\R$ given by 
\begin{equation} \label{eq:PhiTheta}
  \phi_{\theta}(x,y)=\log\int_0^1a^{x+1}(1-a)^{y}\dd\nu_{\theta}(a).
\end{equation}
The criterion function $\t \mapsto\ell_n(\t)$ is defined as 
\begin{equation}\label{eq:l}
\ell_n(\theta)=\sum_{x=0}^{n-1}
 \phi_{\theta}(L_{x+1}^n,L_{x}^n),
\end{equation}
that is the first sum (dominant term) in (\ref{eq:vrsbl2}). 
\end{defi}
We  maximise this  criterion function  to obtain  an estimator  of the
unknown parameter. 
To prove
convergence of the estimator,  some assumptions
are further required.
\begin{assu}(Ballistic case). \label{assu:ballistic}
For any $ \theta\in\Theta, \   \E^{\t}(\rho_0)<1$.
 \end{assu}
As already mentioned, Assumption~\ref{assu:trans} is equivalent to the
transience of the 
walk to the right, and together with
Assumption~\ref{assu:ballistic}, it implies positive speed.
\begin{assu}(Continuity). \label{assu:cont}
For  any $x,y\in  \N$, the  map $\t  \mapsto
\phi_{\theta}(x,y)$ is continuous on $\Theta$.
\end{assu}
Assumption~\ref{assu:cont}  is  equivalent  to  the  map  $\t  \mapsto
\nu_\t$ being continuous on $\Theta$ with respect to the weak topology.

 \begin{assu}(Identifiability).\label{assu:ident}
 $\forall (\theta,\theta')\in\Theta^2,$ 
 $\nu_{\theta}\neq\nu_{\theta'}\iff \theta\neq\theta'.$
 \end {assu}
 
\begin{assu} \label{assu:BoundedBelow}
The collection  of probability  measures $\{ \nu_\t  \, : \,  \t \in \Theta \}$ is such that
\[
   \inf_{\theta \in \Theta} \E^\t [\log (1- \w_0)] > -\infty.
\]
\end{assu}
Note that under Assumption~\ref{assu:ballistic} we have $\E^\t [\log \w_0] > -\log 2$ for any $\t \in \Theta$. Assumptions~\ref{assu:cont} and~\ref{assu:BoundedBelow} are  technical and involved in  the proof of the consistency of our estimator. Assumption~\ref{assu:ident} states identifiability of the  parameter~$\theta$ with
respect  to  the   environment  distribution  $\nu_\theta$  and  is
necessary for estimation. 

According  to Assumption~\ref{assu:cont},  the  function $\theta\mapsto
\ell_n(\theta)$   is   continuous    on   the   compact   parameter   set
$\Theta$.  Thus,   it achieves its maximum, and we define the estimator
$\htet$ as a maximiser. 
\begin{defi}
An estimator $ \htet$ of $\theta$ is defined as a measurable choice
\begin{equation}\label{eq:estimator}
\htet \in \argmax_{\theta\in\Theta}\ell_n(\theta).
\end{equation}
\end{defi}
Note that $\htet$ is not necessarily unique. 
\begin{rema}
  The estimator $\htet$  is a  $M$-estimator,  that  is,
 the maximiser of some criterion function of the observations.  
 The criterion $\ell_n$ is not exactly the log-likelihood for we neglected the contribution of the 
 negative sites. However, with  some abuse  of notation,  we call  $\htet$ a
  maximum likelihood estimator. 
\end{rema}


\subsection{Asymptotic consistency of the estimator in
  the ballistic case}\label{sect:consistency_statement}
From now  on, we assume  that the process  $X$ is generated  under the
true parameter value  $\ts$, an interior point of  the parameter space
$\Theta$, that we want to estimate. We shorten to 
$\PPs$   and   $\EEs$   (resp.   $\Ps$   and   $\Es$)   the   annealed
(resp. quenched) probability $\PP^{\ts}$ (resp. $\P^{\ts}$) and
corresponding expectation~$\EE^{\ts}$ (resp. $\E^\ts$) under
parameter value $\ts$.

\begin{theo}(Consistency). \label{theo:consistency}
Under Assumptions~\ref{assu:trans} to~\ref{assu:BoundedBelow}, 
for any  choice of $\widehat\theta_n$  satisfying~\eqref{eq:estimator}, we
have
\[
  \lim_{n\to\infty}\widehat\theta_n=\ts,
\]
in $\PPs$-probability.
\end{theo}


\section{Examples} \label{sect:ex}

\subsection{Environment with finite and known support} \label{sect:ex1}

\begin{exam} \label{ex:deuxpoints}
Fix $a_1 < a_2 \in (0,1)$ and let
$\nu=p\delta_{a_1}+(1-p)\delta_{a_2}$, where  $\delta_a$ is the Dirac
mass located at $a$. Here, the unknown parameter is the proportion
$p\in  \Theta  \subset  [0,1]$  (namely~$\theta=p$). We  suppose  that
$a_1$, $a_2$  and $\Theta$ are  such that Assumptions~\ref{assu:trans}
and~\ref{assu:ballistic} are satisfied. 
\end{exam}

This example is easily generalised to $\nu$ having $m\ge 2$ support points namely
$  \nu=\sum_{i=1}^m p_i\delta_{a_i}$, 
where $a_1,\dots,a_m$ are distinct, fixed and known in $(0,1)$, we let
$p_m=1-\sum_{i=1}^{m-1}p_i$     and    the     parameter     is    now
$\theta=(p_1,\dots,p_{m-1})$. 

In the framework of Example~\ref{ex:deuxpoints}, we have
\begin{equation} \label{equa:Phi2pts}
 \phi_{p}(x,y)      =\log    [    p    a_1^{x+1}(1-a_1)^y    +(1-p)
 a_2^{x+1}(1-a_2)^y ] ,
\end{equation}
and
\begin{equation}\label{eq:ell_2points}
  \ell_n(p):=\ell_n(\theta) = \sum_{x=0}^{n-1}\log \Big[ p
  a_1^{L_{x+1}^n+1}(1 \!  -\! a_1)_{\phantom{1}}^{L_x^n} + (1\! -\! p) 
  a_2^{L_{x+1}^n+1}(1\! -\! a_2)_{\phantom{2}}^{L_x^n} \Big].
\end{equation}

Now, it is easily seen that Assumptions~\ref{assu:cont} to~\ref{assu:BoundedBelow} are satisfied.  
Coupling this point with the concavity of the
function $p\mapsto \ell_n(p)$ implies that 
$
 \widehat p_n = \argmax_{p\in\Theta}\ell_n(p)
$ 
is well-defined and unique (as $\Theta$ is a compact set). 
There is no  analytical expression for the
value  of  $\widehat  p_n$.  Nonetheless,  this estimator  may  be  easily
computed by numerical methods. Finally, it is consistent from Theorem~\ref{theo:consistency}.

\subsection{Environment with two unknown support points} \label{sect:ex2}

\begin{exam} \label{ex:deuxpts_3param}
We let $\nu=p\delta_{a_1}+(1-p)\delta_{a_2}$  and now the unknown
  parameter is  $\theta=(p,a_1,a_2) \in  \Theta$, where $\Theta$  is a
  compact subset of
\[
  (0,1)\times \{(a_1,a_2)\in (0,1)^2 \, : \,  a_1 < a_2\}.
\]
We  suppose that  $\Theta$ is  such  that Assumptions~\ref{assu:trans}
and~\ref{assu:ballistic} are satisfied. 
\end{exam}

This case is particularly interesting  as it corresponds to one of the
setups in the DNA unzipping experiments,
namely estimating binding energies with two types
of interactions: weak or strong.

The  function $\phi_\t$  and the  criterion  $\ell_n(\cdot)$ are
given        by~\eqref{equa:Phi2pts}        and~\eqref{eq:ell_2points},
respectively.  It  is  easily  seen  that  Assumptions~\ref{assu:cont}
to~\ref{assu:BoundedBelow} are  satisfied in  this setup, so  that the
estimator $ \widehat \t_n $ is well-defined. 
Once again, there is no  analytical expression for the
value  of  $\widehat  \t_n$.  Nonetheless,  this estimator  may  also be  easily
computed       by       numerical       methods.       Thanks       to
Theorem~\ref{theo:consistency}, it is consistent.

\subsection{Environment with Beta distribution} \label{sect:ex3}

\begin{exam} \label{ex:beta}
We let $\nu$ be  a Beta distribution with parameters $(\alpha,\beta)$,
namely
\[
\dd\nu(a) = \frac 1 {\Beta(\alpha,\beta)} a^{\alpha
  -1}(1-a)^{\beta -1} \dd a, \quad
   \Beta(\alpha,\beta)   =   \int_0^1   t^{\alpha
  -1}(1-t)^{\beta -1}\dd t.
\]
Here,  the  unknown parameter  is  $\theta=(\alpha,\beta) \in  \Theta$
where $\Theta$ is a compact subset of 
\[
  \{ (\alpha,\beta) \in (0,+ \infty)^2 \, : \, \alpha > \beta+1 \}.
\]
As $\E^{\t}(\rho_0)=\beta/(\alpha-1)$, the  constraint $\alpha > \beta
+1$ ensures that Assumptions~\ref{assu:trans} and~\ref{assu:ballistic}
are       satisfied.
\end{exam}

In the framework of Example~\ref{ex:beta}, we have
\begin{equation}
\phi_\t(x,y) = \log \frac{\Beta(x+1+\alpha,y+\beta)}{\Beta(\alpha,\beta)}
\end{equation}
and
\begin{align*} \label{eq:ell_beta}
  \ell_n(\theta) &= -n  \log \Beta(\alpha,\beta)+ \sum_{x=0}^{n-1} \log
  \Beta(L_{x+1}^n + \alpha+1,L_x^n + \beta)\\
&=   
 \sum_{x=0}^{n-1} \log \frac{(L_{x\!+\!1}^n \!+ \!\alpha)(L_{x\!+\!1}^n\! +\! \alpha\!-\!1)\ldots \alpha
 \times (L_{x}^n \!+\! \beta\!-\!1)(L_{x}^n \!+\! \beta\!-\!2)\ldots \beta}
 {(L_{x+1}^n \!+\!L_{x}^n \!+\! \alpha \!+\!\beta\!-\!1)(L_{x+1}^n\! +\!L_{x}^n \!+\! \alpha \!+\!\beta\!-\!2)\ldots
 (\alpha+ \beta)}.
 \end{align*}
In this case, it is easily seen that Assumptions~\ref{assu:cont}
to~\ref{assu:BoundedBelow} are satisfied, ensuring that $ \widehat \t_n $ 
is well-defined. Moreover, thanks to Theorem~\ref{theo:consistency}, it
is consistent.


\section{Consistency} \label{sect:consistency}

The proof of Theorem~\ref{theo:consistency} relies on classical theory
about the convergence of  maximum likelihood estimators, as stated for
instance in the classical approach by~\cite{Wald} for i.i.d.  random
variables. 
We  refer for  instance to  Theorem  5.14 in~\cite{VdV}  for a  simple
presentation of Wald's approach and further stress that the proof is valid on a compact
parameter space only. 
It relies on the two following ingredients.
\begin{theo} \label{thm:LimitEllTheta}
Under  Assumptions~\ref{assu:trans}  to~\ref{assu:BoundedBelow}, there
exists a finite deterministic limit $\ell(\t)$ such that 
\begin{equation*}
  \frac 1n \ell_n(\t) \xrightarrow[n \to \infty]{} \ell(\t) \quad
\mbox{in } \PPs \mbox{-probability},
\end{equation*}
and this  convergence is  "locally uniform" with  respect to $\t$.
\end{theo}
The sense of the local uniform convergence is specified in Lemma~\ref{lemm:cv_du_sup} in Subsection~\ref{sect:UnifConv}, and the value of $\ell(\t)$ is given in (\ref{eq:ell}).

\begin{prop} \label{prop:IdentificationEllTheta}
Under Assumptions~\ref{assu:trans} to~\ref{assu:BoundedBelow}, for any $\eps >0$,
\[
\sup_{\theta: \|\theta-\ts\| \geq \eps} \ell(\theta)<\ell(\ts).
\]
\end{prop}
Theorem~\ref{thm:LimitEllTheta}  induces a  pointwise convergence
of the normalised criterion~$\ell_n/n$ 
to some limiting function $\ell$,  and is weaker than assuming uniform
convergence.  Proposition~\ref{prop:IdentificationEllTheta}
states that  the former limiting  function $\ell$ identifies  the true
value of the 
parameter   $\ts$,  as  the   unique  point   where  it   attains  its
maximum.     

Here    is    the    outline    of    the    current    section.    In
Subsection~\ref{sect:BP}, we  recall some preliminary  results linking
RWRE  with  branching  processes  in  random  environment  (BPRE).  In
Subsection~\ref{sect:LLN},  we  define  the limiting  function  $\ell$
involved  in Theorem~\ref{thm:LimitEllTheta} thanks  to a  law of
large numbers  (LLN) for Markov chains. 
In  Subsections~\ref{sect:UnifConv}   and~\ref{sect:Ident},  we  prove
Theorem~\ref{thm:LimitEllTheta} and Proposition~\ref{prop:IdentificationEllTheta}, respectively.  
 It is important to note that the limiting function $\ell$ exists as soon as the walk is
transient. However, it is finite in the ballistic case and everywhere 
infinite in 
the sub-ballistic regime of uniformly elliptic walks, see Lemma~\ref{lemm:ue}. 
This latter
fact     prevents    the     identification    result     stated    in
Proposition~\ref{prop:IdentificationEllTheta}  and   explains  why  we
obtain consistency only in the ballistic regime. 
From all  these ingredients, the consistency of  $\widehat \t_n$, that
is,    the     proof    of    Theorem~\ref{theo:consistency}    easily
follows.


\subsection{From RWRE to branching processes}\label{sect:BP}
We start  by recalling  some already known  results linking  RWRE with
branching processes in random  environment (BPRE). Indeed, it has been
previously observed in \cite{KKS} that 
for    fixed   environment   $\w=\{\w_x\}_{x\in\Z}$,    under   quenched
distribution $P_{\w},$ the sequence $ L_n^n, L_{n-1}^n,\dots, L^n_0$ 
of the number of left steps performed by the process $X_{[0,T_n]}$ from
sites $n, n-1,\dots, 0$, 
has  the  same  distribution  as  the  first  $n$  generations  of  an
inhomogeneous branching process with  one immigrant at each generation
and with geometric offspring.  

More precisely, for  any fixed  value $n\in  \N^*$ and
fixed environment $\w$, consider a family of independent random
variables 
$\{ \xi_{k,i} \, : \, k\in \{1,\dots,n\}, \, i\in\N \}$ 
such that for each fixed value
$k\in \{1,\dots,n\}$, the $\{\xi_{k,i}\}_{i\in\N}$  are i.i.d. with a geometric
distribution on $\N$ of parameter $\w_{n-k}$, namely 
\[
  \forall m\in\N, \quad P_{\w}(\xi_{k,i}=m)=(1-\w_{n-k})^m\w_{n-k}.
\]
Then,   let   us   consider   the   sequence   of   random   variables
$\{Z_k^n\}_{k=0,\dots, n}$ defined recursively by 
\[
  Z_0^n=0,\quad\mbox{and for}\ k=0,\dots,n-1, \quad
Z_{k+1}^n=\sum_{i=0}^{Z_k^n}\xi_{k+1,i} .
\]
The sequence $\{Z_k^n\}_{k=0,\dots, n}$ forms an inhomogeneous BP
with immigration  (one immigrant  per generation corresponding  to the
index $i=0$ in the above sum) and whose offspring law depends on
$n$ (hence  the superscript $n$  in notation $Z_k^n$).  Then,  we obtain
that 
\[ 
  (L_n^n,L_{n-1}^n,\dots,L_0^n) \sim (Z_0^n,Z_1^n,\dots,Z_n^n), 
\] 
where $\sim$ means equality in distribution.
 When the environment is random as well, and since $(\w_0,\dots,\w_n)$ has the same 
 distribution as $(\w_n,\dots,\w_0)$,  it follows that under  the annealed
 law  $\PPs$, the  sequence $L_n^n,L_{n-1}^n,\dots,  L^n_0$ has  the same
 distribution  as a  branching  process in  random environment  (BPRE)
 $Z_0,\dots, Z_n,$ 
defined by 
\begin{equation} \label{eq:BPRE}
Z_0=0, \quad \mbox{and for } k=0,\dots,n,\quad
Z_{k+1}=\sum_{i=0}^{Z_k}\xi'_{k+1,i},
\end{equation}
with $\{\xi'_{k,i}\}_{k\in\N^*;i\in\N}$ independent and
\[
  \forall m\in\N, \quad P_{\w}(\xi'_{k,i}=m)=(1-\w_{k})^m\w_{k} . 
\]
Now,  when  the
  environment is assumed to be i.i.d., this BPRE is under annealed law
  a homogeneous Markov chain.  We explicitly state this result because it is
  important; however its proof is immediate and therefore omitted. 
  
\begin{prop}\label{prop:trans}
Suppose   that  $\{\w_n\}_{n\in\N}$   are  i.i.d.   with  distribution
$\nu_{\t}$.  Then the sequence $\{Z_n\}_{n\in\N}$ is a homogeneous Markov chain
whose    transition   kernel  $Q_{\t}$    is   given    by
\begin{equation}
  \label{eq:Qtransition}
Q_{\t}(x,y)=\binom{x+y}{x}\ee^{\phi_\t(x,y)}=  \binom{x+y}{x} \int_0^1
a^{x+1}(1-a)^y \dd\nu_{\t}(a) .
\end{equation}
 \end{prop}

Finally,       going      back      to~\eqref{eq:l}       and      the
definition~\eqref{eq:PhiTheta} of $\phi_{\t}$, the annealed log-likelihood satisfies the following equality 
\begin{equation}\label{eq:logvraisannealed}
\ell_n(\theta)
\sim \sum_{k=0}^{n-1}\phi_{\theta}(Z_k,Z_{k+1}) \mbox{ under } \PPs.
\end{equation}
\begin{rema}
Up to an additive constant (not depending on $\theta$), the right-hand
side of~\eqref{eq:logvraisannealed}  is exactly the  log-likelihood of
the Markov chain $\{Z_k\}_{0\le k \le n}$.
Indeed, we have
\[
\log Q_{\t}(x,y)=\log \binom{x+y}{x} +\phi_{\t}(x,y) , \quad \forall x,y\in \N.
\]
\end{rema}
We prove in the next section a weak law of large numbers for 
$\{\phi_\t(Z_k,Z_{k+1})\}_{k \in \N}$ and according to~\eqref{eq:logvraisannealed}, this is sufficient to obtain a weak convergence of $\ell_n(\t) / n$.

\subsection{Existence of a limiting function} \label{sect:LLN}
It   was    shown   by   \citet[Theorem    3.3,][]{Key}   that   under
Assumption~\ref{assu:trans} (and for a non-necessarily 
i.i.d. environment),  the sequence  $\{Z_n\}_{n \in \N}$  converges in
annea\-led law to a limit 
random  variable  $\tilde  Z_0$  which  is almost  surely  finite.  An
explicit construction  of $\tilde Z_0$  is given by Equation  (2.2) in
\cite{Ro}. In fact, a complete \emph{stationary version} $\{\tilde Z_n\}_{n
  \in \Z}$  of the sequence $\{Z_n\}_{n  \in \N}$ is given  and such a
construction allows for an ergodic theorem.  In the i.i.d. environment
setup, we obtain more precise results than what is provided by~\citet[Theorem 3.3,][]{Key}, as $\{Z_n\}_{n \in \N}$ is a Markov chain. 
Thus Theorem~\ref{theo:limitlaw} below is specific to our setup: geometric offspring distribution, one immigrant per
generation and i.i.d.  environment. 
We  specify the  form of  the  limiting distribution  of the  sequence
$\{Z_n\}_{n \in \N}$ and characterise its first moment. We later rely on these results to establish a strong law of large numbers for the sequence $\{\phi_\t(Z_k,Z_{k+1})\}_{k\in \N}$.

\begin{theo}\label{theo:limitlaw} 
Under Assumption~\ref{assu:trans}, for all $\theta \in\Theta$ the following assertions hold
 \begin{itemize}
\item [i)] The Markov  chain $\{Z_n\}_{n\in \N}$ is positive recurrent
  and admits a unique invariant
probability measure $\pi_{\t}$ satisfying
\[  \lim_{n \rightarrow \infty} \PPt(Z_n = k) = \pi_{\t}(k), \quad \forall k \in \N .\] 
\item [ii)] 
Moreover, for all $k\in \N$, we have
  $\pi_{\t}(k)=\Et[S(1-S)^k]$, where 
\[
S := (1+\rho_1 + \rho_1\rho_{2}+\dots + \rho_1\dots\rho_n+\dots)^{-1}\in (0,1). 
\]
In particular, we have $\sum_{k\in \N} k\pi_{\t}(k)=\sum_{n \ge 1} (\Et
\rho_0)^n$, and the distribution~$\pi_{\t}$ has a finite first order
moment only in the ballistic case.
\end{itemize}
 \end{theo} 

\begin{proof}
We introduce the quenched probability generating function of the random
variables   $\xi'_{n,i}$  and~$Z_n$   introduced  in~\eqref{eq:BPRE},
respectively defined for any $u \in [0,1]$ by 
\begin{equation*}
   H_n(u) := E_{\w}\left(u^{ \xi'_{n,0}}\right)=\frac{\w_{n}}{1-(1-\w_{n})u}, \quad 
\mbox{and} \quad F_n(u) := E_{\w}\left(u^{ Z_n}\right),
\end{equation*}
as well as the quantities $S_n$ and $\tilde S_n$ defined as
\begin{align*} 
S_n^{-1}&=1+\rho_n + \rho_n\rho_{n-1}+\dots + \rho_n\dots\rho_1 ,\\
\tilde S_n^{-1}&=1+\rho_1 + \rho_1\rho_{2}+\dots +\rho_1\dots\rho_n .
\end{align*}
According to~\eqref{eq:BPRE}, we have
\[
  F_{n+1}(u)=F_n[H_{n+1}(u)]\times H_{n+1}(u) ,
\]
and by induction 
\[
  F_n(u)=\frac{\w_1\dots \w_n}{A_n(\w)-B_n(\w)u},
\]
   where $A_n$ and $B_n$ satisfy the relations 
\[
  \left \{\begin{array}{l}
A_{n+1}(\w)=A_n(\w)-B_n(\w)\times\w_{n+1} , \\
B_{n+1}(\w)=A_n(\w)\times(1-\w_{n+1}) . 
\end{array}\right.
\]
A simple computation yields
\[
  A_{n}(\w)=\w_1\dots\w_n\,S_n^{-1} 
\quad \mbox{ and } \quad 
  B_{n}(\w)=\w_1\dots\w_n(S_n^{-1}-1).
\]
Finally, we have for any $u \in [0,1]$
\[
  F_n(u)=\frac{S_n}{1-(1-S_n)u}.
\]
This   means that  under quenched  law $P_{\w}$,  the  random variable
$Z_n$ follows a geometric distribution on $\N$ with parameter $S_n$. 
Note  that $S_n$  and $\tilde  S_n$ have  the same  distribution under
$\Pt$, implying that $F_n(u)$ has the same distribution as 
\[
  \frac{\tilde S_n}{1-(1-\tilde S_n)u}.
\]
Under Assumption~\ref{assu:trans}, we have $\Pt$-a.s.
\[
  \lim_{n\to\infty}\frac                     1n                    \log
(\rho_1\dots\rho_n)=\lim_{n\to\infty}\frac
1n\sum_{i=1}^n\log\rho_i=\Et\log\rho_0:=m <0,
\]
and  hence 
\[
  \Pt\left(\exists  \ n(\w),\  s.t. \  \forall  n>n(\w), \
  \rho_1\dots \rho_n\leq \ee^{nm/2}\right)=1.
\]
Then, as $n\to  +\infty$, $\tilde S_n \searrow S=(1+\rho_1 + \rho_1\rho_{2}+\dots )^{-1}$ $\Pt$-a.s. 
 with $\Pt(0 < S < 1 )=1$.  As  a
consequence,  the  quenched   probability  generating  function  $F_n(u)$
converges in distribution under $\Pt$ to 
\[
  F(u)=\frac{S}{1-(1-S)u},
\]
the probability  generating function of a  geometric distribution with
parameter $S$. 
Under annealed law, for any $k\in \N$ we have

\[
\PPt(Z_n=k)=\Et  P_{\w}(Z_n=k)=\Et  \left[  {S_n}  \left(  1  -  {S_n}
  \right)^k \right] = \Et \left[ \tilde S_n \left( 1 - \tilde S_n \right)^k \right].
\]
Since
 $0<\tilde  S_n<1$, dominated convergence implies that for all $k\in\N,$
\begin{equation}
  \label{eq:limite_Zn}
 \lim_{n\to +\infty} \PPt(Z_n=k) =
\Et \left[ S \left(1 - {S}\right)^k \right] := \pi_{\t}(k) .
\end{equation}
As an immediate consequence, we obtain
\[
\sum_{k\in \N} k\pi_{\t}(k)=\Et \left( S^{-1}-1
\right)=\sum_{n=1}^{\infty} (\Et \rho_0)^n ,
\] 
Moreover,   by    Fubini-Tonelli's  theorem   and
$\Pt(0<S<1)=1$, we have 
\[
\sum_{k\in \N} \pi_{\t}(k) = 1 \quad \mbox{ and } \quad \pi_{\t}(k)>0,
\quad \forall k \in \N.
\]
Thus the measure $\pi_{\t}$ on $\N$ is a probability measure and thanks 
to~\eqref{eq:limite_Zn}, it is invariant.
We
note that $\{Z_n\}_{n \in \N}$ is irreducible   as the
transitions  $Q_{\t}(x,y)$ defined  by \eqref{eq:Qtransition}  are positive
and the  measure $\nu_{\t}$  is not degenerate.   Thus, the  chain is
positive    recurrent   and    $\pi_{\t}$   is    unique    \cite[see   for
instance][Theorem 1.7.7]{Norris}.
%
This concludes the proof. 
\end{proof}

 Let us define $\{\tilde Z_n\}_{n\in \N}$ as the stationary  Markov chain with
 transition matrix  $Q_{\ts}$ defined by~\eqref{eq:Qtransition}  and initial
 distribution $\pi^\star:=\pi_{\t^\star}$ introduced in Theorem~\ref{theo:limitlaw}. It 
 will not be confused with  $\{Z_n\}_{n\in \N}$ from \eqref{eq:BPRE}. 
 We let 
$\ell(\t)$ be defined as 
\begin{equation}
  \label{eq:ell}
\ell(\theta)=\EEs[\phi_{\theta}(\tilde Z_0,\tilde Z_1)] \in [-\infty,0],
\end{equation}
where $\phi_\theta$ is defined according to~\eqref{eq:PhiTheta}. 
(Note that the quantity $\ell(\theta)$ may not necessarily be finite). 
As a consequence of the  
irreducibility of the chain $\{Z_n\}_{n\in \N}$
 and  Theorem~\ref{theo:limitlaw},  we  obtain the  following  ergodic
 theorem \cite[see for instance][Theorem 1.10.2]{Norris}. 

\begin{prop}\label{prop:lln}
Under Assumption~\ref{assu:trans}, for all $\theta\in\Theta$, 
the following ergodic theorem holds 
\[
  \lim_{n\to\infty}\frac1n
\sum_{k=0}^{n-1}\phi_{\theta}(Z_k,Z_{k+1})
=\ell(\theta) \quad \PPs\mbox{-almost surely}.
\]
\end{prop}

\subsection{Local uniform convergence and finiteness of the limit} 
\label{sect:UnifConv} 
According to~\eqref{eq:logvraisannealed} and Proposition~\ref{prop:lln}, we obtain 
\begin{equation} \label{eq:cv_loglik}
\lim_{n\to\infty}\frac1n     \ell_n(\theta)     =    \ell(\theta)\quad
\mbox{in} \quad \PPs \mbox{-probability}.
\end{equation}
To  achieve  the  proof  of  Theorem~\ref{thm:LimitEllTheta},  it
remains  to prove  that the    convergence    is    "locally    uniform" and that  the limit  $\ell(\theta)$  is finite  for  any value  of
$\t$.     The local  uniform convergence is  given by
Lemma~\ref{lemm:cv_du_sup} below while Proposition~\ref{prop:ell_fini} gives  a sufficient condition  for the latter fact to occur.

\begin{lemm}\label{lemm:cv_du_sup}
Under   Assumption~\ref{assu:trans},  the   following   local  uniform
convergence holds: for  any   open   subset  $U\subset
\Theta$, 
\[
\frac  1n   \sum_{x=0}^{n-1}\sup_{\t\in U} \phi_{\t} (L_{x+1}^n,L_x^n)
\xrightarrow[n \to \infty]{} \EEs \Big( \sup_{\t\in U } \phi_\t (\tilde Z_0,\tilde Z_1) \Big) \quad
\mbox{in } \PPs \mbox{-probability } .
\]
\end{lemm}

\begin{proof}[Proof of Lemma~\ref{lemm:cv_du_sup}]
Let us fix an open subset $U\subset \Theta$ and note that 
\begin{equation*}
\frac  1n   \sum_{x=0}^{n-1}
\sup_{\t\in U} \phi_{\t} (L_{x+1}^n,L_x^n) \sim \frac 1n \sum_{k=0}^{n-1}
\Phi_U (Z_k,Z_{k+1}) ,
\end{equation*}
where we have $\Phi_U:= \sup_{\t\in U}\phi_{\t}$.  As the function $\Phi_U$ is non-positive,
the expectation~$\EEs( \Phi_U (\tilde Z_0,\tilde Z_{1}))$ exists and
relying again on the ergodic theorem for Markov chains, we obtain the desired
result.  
\end{proof}

\begin{prop}\label{prop:ell_fini}
(Ballistic case). As soon as
  \begin{equation} \label{eq:ballistique}
   \Es(\rho_0)<1,
  \end{equation}
the limit $\ell(\theta)$ is
finite for  any value $\theta \in \Theta$.  
\end{prop}

\begin{proof}[Proof of Proposition~\ref{prop:ell_fini}]
  For all $x\in\N, y\in\N$, by using Jensen's inequality, we may write 
\begin{equation} \label{equa:MinoPhiTheta}
\log\int_0^1a^{x+1}(1-a)^{y}\dd\nu_{\theta}(a) \geq 
(x+1)\Et [\log(w_0)] +y \Et [\log(1- w_0)] .  
\end{equation}
This implies that for any $k\in \N$,
\begin{equation*}
\phi_{\theta}(Z_k,Z_{k+1}) \geq 
(Z_k+1)\E^{\t} [\log(w_0)] +Z_{k+1} \E^{\t} [\log(1- w_0)],   
\end{equation*}
and in particular 
\begin{equation}\label{equa:MinoEllTheta}
\frac    1n\sum_{k=0}^{n-1}\phi_{\theta}(Z_k,Z_{k+1})    \geq    
\E^{\t} [\log(w_0)]  \frac 1n\sum_{k=0}^{n-1}(Z_k+1) +  \E^{\t} [\log(1- w_0)]
\frac 1n\sum_{k=0}^{n-1}Z_{k+1} .
\end{equation}
Now, as a consequence of Theorem~\ref{theo:limitlaw}, we know that in the
ballistic   case  given   by~\eqref{eq:ballistique}   the  expectation
$\EEs(\tilde Z_0)$  is finite.  From the ergodic theorem, $\PPs$-almost surely, 
\begin{equation} \label{eq:cv_sum_Zk}
\frac  1n\sum_{k=0}^{n-1}(Z_k+1)  \xrightarrow[n  \to  \infty]{}  \EEs
(\tilde Z_0) +1 \quad 
\text{and}  \quad \frac  1n\sum_{k=0}^{n-1}Z_{k+1}  \xrightarrow[n \to
\infty]{} \EEs (\tilde Z_0). 
\end{equation}
Combining this convergence  with the lower bound in~\eqref{equa:MinoEllTheta},
we obtain $\ell(\t)\in (-\infty,0]$ in this case.
\end{proof}

The next lemma specifies that
condition~\eqref{eq:ballistique}  is necessary  for  $\ell(\t)$ to  be
finite at least in a particular case.

\begin{lemm}\label{lemm:ue}
(Converse result in the uniformly elliptic case).
Assume that  $\nu_\t([\delta, 1-\delta])=1$  for some $\delta  >0$ and
all  $\theta  \in \Theta$  (uniformly  elliptic  walk).  Then, in  the
sub-ballistic case, that is $\Es(\rho_0) \geq1$, the limit  $\ell(\theta)$ is
infinite for all parameter values. 
\end{lemm}

\begin{proof}
For  any integers $x$  and $y$ and  any $a$ in the  support of
$\nu_\t$, we have 
\[
0<  \delta^{x+1}\leq  a^{x+1}\leq  (1-\delta)^{x+1}  ,  \quad\quad  0<
\delta^y\leq (1-a)^y\leq (1-\delta)^y,
\]
and then
\begin{equation*}
(x+y+1)\log(\delta)                                                 \leq
\log\int_0^1a^{x+1}(1-a)^{y}\dd \nu_{\theta}(a)   \leq   (x+y+1)\log
(1-\delta). 
\end{equation*}
This implies that for any $k\in \N$,
\begin{equation*}
(Z_k+Z_{k+1}+1)\log          (\delta)\leq\phi_{\theta}(Z_k,Z_{k+1})\leq
(Z_k+Z_{k+1}+1)\log (1-\delta),
\end{equation*}
and in particular
\begin{equation}\label{equa:MajoEllTheta}
\frac 1n\sum_{k=0}^{n-1}\phi_{\theta}(Z_k,Z_{k+1}) 
\leq \log (1-\delta)\frac 1n\sum_{k=0}^{n-1}(Z_k+Z_{k+1}+1).
\end{equation}
Combining  the convergence~\eqref{eq:cv_sum_Zk}  with the  above upper
bound  implies   that  as  soon  as   $\ell(\theta)>-\infty$,  we  get
$\EEs(\tilde   Z_0)<+\infty$  and   according  to   point   $iii)$  in
Theorem~\ref{theo:limitlaw},   this  corresponds   to   the  ballistic
case~\eqref{eq:ballistique}.  
\end{proof}


\subsection{Identification of the true parameter value} \label{sect:Ident}
Fix     $\eps    >0$.     We    want     to    prove     that    under
Assumptions~\ref{assu:trans} to~\ref{assu:BoundedBelow},
\[
  \sup_{\theta : \|\theta -\theta^\star\| \geq \eps} \ell(\t) < \ell(\theta^\star).
\] 
First of all,  note that according to Proposition~\ref{prop:ell_fini},
Assumption~\ref{assu:ballistic} ensures  that $\ell(\t)$ is  finite for
any value $\t \in \Theta$. 

Now, we start by proving that for any $ \theta\in\Theta,$ we have $\
\ell(\theta)\leq \ell(\ts).$ According to~\eqref{eq:ell}, we may write 
\begin{equation*}
  \ell(\theta)-\ell(\ts)      =     \EEs[\phi_\t(\tilde     Z_0,\tilde
  Z_1)-\phi_{\ts}(\tilde     Z_0,\tilde Z_1) ],
\end{equation*}
which may be rewritten as
\[
  \sum_{x\in\N}\pi^\star(x)\left      [\sum_{y\in\N}\log\left
    (\frac{Q_\t(x,y)}{ Q_{\ts}(x,y)}\right )Q_{\ts}(x,y)\right ].
\]
Using  Jensen's inequality  with  respect  to the logarithm
function and the (conditional) distribution $Q_{\ts}(x,\cdot)$ yields 
\begin{equation}\label{equa:Jensen}
\ell(\theta)-\ell(\ts) \leq 
\sum_{x\in\N}\pi^\star(x) 
        \log                                                       \Big
        (\sum_{y\in\N}\frac{Q_\t(x,y)}{Q_{\ts}(x,y)       }Q_{\ts}(x,y)
        \Big) =0.
\end{equation}
The equality in~\eqref{equa:Jensen}  occurs if and only if  for any $x
\in \N$, we have $Q_\t(x,\cdot)=  Q_{\ts}(x,\cdot), $
which  is   equivalent  to  the  probability   measures  $\nu_\t$  and
$\nu_{\ts}$
having identical  moments.  Since their  supports are included  in the
bounded set 
$(0,1)$, these  probability measures are then  identical \cite[see for
instance][Chapter  II,  Paragraph 12,  Theorem  7]{Shir}.  Hence,  the
equality $\ell(\t) =  \ell(\ts)$ yields $\nu_{\theta}=\nu_{\ts}$ which
is equivalent to $\t=\ts$ from Assumption~\ref{assu:ident}.  

In   other   words,  we   proved   that
$\ell(\t)\le  \ell(\ts)$ with  equality if  and only  if  $\t=\ts$. To
conclude  the proof  of Proposition~\ref{prop:IdentificationEllTheta},
it suffices to establish that the function $\t 
\mapsto \ell(\t)$ is continuous. 

From Inequality~\eqref{equa:MinoPhiTheta} and
Assumption~\ref{assu:BoundedBelow}, 
we know  that there exists a  positive constant $A$ such  that for any
$\theta \in \Theta$,
\[
   \left| \phi_\t(\tilde Z_0,\tilde  Z_1) \right| \leq A (  1 + \tilde
   Z_0+ \tilde Z_1).
\]
Under  Assumption~\ref{assu:ballistic},   we  know  that  $\EEs(\tilde
Z_0)= \EEs(\tilde Z_1)$ is finite, and under Assumption~\ref{assu:cont}, the
function  $\t  \mapsto  \phi_\t(x,y)$   is  continuous  for  any  pair
$(x,y)$.  We  deduce  that  the  function  $\t  \mapsto  \ell(\t)$  is
continuous.




\section{Numerical performances}
\label{sect:simus}

In this section, we explore the numerical performances of our 
estimation procedure and compare them with the performances  of the
estimator  proposed  by      \cite{AdEn}. As this latter procedure is rather involved and far more general than ours, we start by describing its form in our specific context in Section~\ref{sect:EstiAdEn}. The simulation protocol as well as corresponding results are given in Section~\ref{sect:experiments}, where we focus on Examples~\ref{ex:deuxpoints} to~\ref{ex:beta}.

\subsection{Estimation procedure of \cite{AdEn}} \label{sect:EstiAdEn}

The  estimator proposed by~\cite{AdEn} is a moment estimator. It  is  based  on  collecting information on  sites displaying some specified  histories. We shortly
explain it in our context: the one dimensional RWRE. 

Let $H(t,x)$ denote the history of site $x$ at time $t$ defined as
\[
  H(t,x) = (L(x,X_{[0,t]}),R(x,X_{[0,t]})),
\]
where $L(x,X_{[0,t]})$  and $R(x,X_{[0,t]})$ are  respectively defined
by~\eqref{equa:LeftStep} 
and~\eqref{equa:RightStep}, and represent the number of left and right
steps performed 
by the walk  at site $x$ until time $t$.  Note that $H(0,x)=(0,0)$ for
any site $x$. 

We define $H(t)$ as the history of the currently occupied site $X_t$ at time $t$,
that is 
\[
  H(t)=H(t,X_t).
\]
For any  $\h=(h_-,h_+) \in \N^2$, let $\{K_i^\h\}_{i\geq 0}$  be the successive
times where the history of the currently occupied site is $\h$:
\[
  K_0^\h  = \inf\{t  \geq 0  \, :  \, H(t)=\h  \}, \quad  K_{i+1}^\h =
  \inf\{t > K_{i}^\h \, : \, H(t)=\h \}.
\]
Define $\Delta_i^\h$ with values in $\{-1,1\}$ as
\[
  \Delta_i^\h = X_{K_{i}^\h+1} - X_{K_{i}^\h},
\]
which represents the move of the walk at time $K_{i}^\h$, that is, the move at the $i$th time where the history of the currently occupied site is $\h$.

According to Proposition 4 and Corollary 2 in~\cite{AdEn}, the random variables $\Delta_i^\h$ are i.i.d. and we have
\begin{align} 
 & \lim_{m \to  \infty} \frac1{m} \sum_{i=1}^m  \1_{\{ \Delta_i^\h=1
    \}} = V_1(\h) \quad \PPs\mbox{-a.s.}, \label{equa:AdEnRight} \\
\mbox{and} \quad 
&  \lim_{m \to  \infty} \frac1{m} \sum_{i=1}^m  \1_{\{ \Delta_i^\h=-1
    \}} = V_{-1}(\h) \quad \PPs\mbox{-a.s.}, \label{equa:AdEnLeft}
\end{align}
where
\[
  V_1(\h)   =       \frac{\Es[\omega_0^{1+h_+             
      }(1-\omega_0)^{h_-}]}{\Es[\omega_0^{h_+}(1-\omega_0)^{h_-}]}      \quad
    \mbox{and} \quad 
  V_{-1}(\h)    =      \frac{\Es[\omega_0^{h_+             
      }(1-\omega_0)^{1+h_-}]}{\Es[\omega_0^{h_+}(1-\omega_0)^{h_-}]}.
\]
The  quantities
  $V_1(\h)$ and $V_{-1}(\h)$ are the annealed right and left transition probabilities
  from the currently occupied site with history  $\h$.  In particular, in our case $V_1(\h)+V_{-1}(\h)=1$. The consequence of the previous convergence result is that by letting the histories $\h$ vary, we can potentially recover all the moments of the distribution $\nu$ and thus this distribution itself. The strategy underlying \citeauthor{AdEn}'s approach is then to estimate some well-chosen moments $V_1(\h)$ or $V_{-1}(\h)$ so as to obtain a set of equations which has to be inverted to recover parameter estimates.  

We thus define $M^\h_n$ and for $\eps=\pm 1$ the estimators $ \widehat V_\eps^n(\h)$ as
\[
  M^\h_n = \sup\{ K_i^\h < T_n \, : \, i \geq 1 \}, \quad 
  \widehat V_\eps^n(\h) = \frac1{M_n^\h} \sum_{i=1}^{M_n^\h} \1_{\{\Delta_i^\h=\eps\}}, \quad \eps=\pm 1.
\]
The quantity $\widehat V_\eps^n(\h)$ is either the proportion of sites from which the first move is to the right ($\eps=1$) or to the left ($\eps=-1$), among those with history $\h$. (In particular, $\widehat V_1^n(\h)+\widehat V_{-1}^n(\h)=1$.) Then, from~\eqref{equa:AdEnRight} and~\eqref{equa:AdEnLeft} and the fact that $T_n$ goes to infinity $\PPs$-almost surely when $n$ grows to infinity, we get
\[
  \lim_{n \to \infty} \widehat V_\eps^n(\h) = V_\eps(\h) \quad \quad \PPs\mbox{-almost surely}.
\]
Hence, we can estimate $\ts$ by the solution of an appropriate system of equations, as illustrated below. 

\paragraph*{Example~\ref{ex:deuxpoints} (continued).} In this case the parameter $\theta$ equals $p$ and we have
\[
  V_1(0,0)   =     \Es[\omega_0] = p^\star a_1 + (1-p^\star)a_2.
\]
Hence, among the visited sites (namely sites with history $\h=(0,0)$), the proportion of those from which the first move is to the right gives an estimator for $p^\star a_1 + (1-p^\star)a_2$. Using this observation, we can estimate $p^\star$.

\paragraph*{Example~\ref{ex:deuxpts_3param} (continued).}
In this  case the parameter  $\theta$ equals $(p,a_1,a_2)$ and  we may
for instance consider 
\begin{align}\label{eq:syst_eq_ex2}
  V_1(0,0)  &=   p^\star a_1^\star + (1-p^\star)a_2^\star, \notag \\
  V_1(0,1) &=  \{p^\star [a_1^\star]^2 + (1-p^\star)[a_2^\star]^2\} \cdot V_1(0,0)^{-1}, \\
  V_1(0,2) &=  \{p^\star [a_1^\star]^3 + (1-p^\star)[a_2^\star]^3\} \cdot V_1(0,1)^{-1}. \notag 
\end{align}
Hence, among  the visited sites  (sites with history  $\h=(0,0)$), the
proportion of those from which the first move is to the right gives an
estimator for  $p^\star a_1^\star +  (1-p^\star)a_2^\star$.  Among the
sites visited at least twice from which the first move is to the right
(sites with  history $\h=(0,1)$), the  proportion of those  from which
the second move  is also to the right gives  an estimator for $p^\star
[a_1^\star]^2 + (1-p^\star)[a_2^\star]^2$.  Among the sites visited at
least three  times from which  the first and  second moves are  to the
right (sites  with history $\h=(0,2)$),  the proportion of  those from
which  the third  move is  also to  the right  gives an  estimator for
$p^\star [a_1^\star]^3 +  (1-p^\star)[a_2^\star]^3$. Using these three
observations, we can theoretically estimate $p^\star$, $a_1^\star$ and
$a_2^\star$, as soon as the solution to this system of three
nonlinear equations is unique. Note that inverting the mapping defined
by~\eqref{eq:syst_eq_ex2} is not trivial. Moreover, while the moment estimators 
might  have small  errors, inverting  the mapping  might result  in an
increase of this error for the parameter estimates.

\paragraph*{Example~\ref{ex:beta} (continued).}
In this case, the parameter $\theta$ equals $(\alpha,\beta)$ and we have
\begin{equation*}
  V_{-1}(0,0) = \frac{\beta^\star}{\alpha^\star+\beta^\star} 
\quad \mbox{and} \quad
  V_{-1}(1,0) =  \frac{\beta^\star+1}{\alpha^\star+\beta^\star+1}.
\end{equation*}
Hence, among the visited sites (sites with history $\h=(0,0)$), the proportion of those from which the first    move   is   to    the   left    gives   an    estimator   for $\frac{\beta^\star}{\alpha^\star+\beta^\star}$.    Among   the   sites
visited at least  twice from which the first move is  to the left (sites with history $\h=(1,0)$), the proportion of  those from which  the second move  is also to  the left gives                 an                 estimator                 for
$\frac{\beta^\star+1}{\alpha^\star+\beta^\star+1}$.   Using  these two
observations, we can estimate $\alpha^\star$ and $\beta^\star$.

\subsection{Experiments}
\label{sect:experiments}
We now present the three simulation experiments corresponding respectively  to Examples~\ref{ex:deuxpoints} to~\ref{ex:beta}. 
Note   that    in   Example~\ref{ex:deuxpts_3param},   the    set   of
Equations~\eqref{eq:syst_eq_ex2}  may  not  be trivially  inverted  to
obtain the parameter $\theta=(p,a_1,a_2)$. In particular, we were
not able to perform (even only numerically) the mapping inversion needed to
compute \citeauthor{AdEn}'s estimator in this case.
Thus, in the experiments presented below, we choose to only consider our estimation procedure in this case. The comparison with \citeauthor{AdEn}'s procedure is given only for Examples~\ref{ex:deuxpoints} and~\ref{ex:beta}. In those cases, while \citeauthor{AdEn}'s procedure may be easily performed, we already obtain much better estimates with our approach. Since inverting the set of Equations~\eqref{eq:syst_eq_ex2} will increase the uncertainty of the moment estimators, we claim that  \citeauthor{AdEn}'s procedure would do even worse in this case.

For  each  of  the  three  simulations, we  \textit{a  priori}  fix  a
parameter  value $\ts$ as  given in  Table~\ref{tabl:theta_values} and
repeat 1 000 times the  procedure described below. We first generate a
random  environment  according  to  $\nu_\ts$  on  the  set  of  sites
$\{-10^4, \dots, 10^4\}$.  In fact,  we do  not use the environment
values  for all the  $10^4$ negative  sites, since  only few  of these
sites are visited by the walk. However the computation cost is very low comparing to the rest of the estimation procedure, and the symmetry is convenient for programming purpose.  Then, we run a random walk in this environment and stop it successively at the hitting times $T_n$ defined by \eqref{equa:HittingTime}, with $n \in \{10^3 k ; 1\le k \le 10 \}$.  For each
stop,  we  estimate $\ts$  according  to  our  procedure and \citeauthor{AdEn}'s one (except for the second simulation). 
In  all three cases,  the parameters  in Table~\ref{tabl:theta_values}
are  chosen such  that  the RWRE  is  transient and  ballistic to  the
right. Note that  the length of the random walk is  not $n$ but rather
$T_n$. This  quantity varies considerably throughout  the three setups
and   the   different   iterations.   Figure~\ref{figu:histos}   shows
(frequency) histograms of  the hitting  times $T_n$ for  some selected  values $n$
($n= $ 1 000, 5 000 and 10 000), obtained from 1 000 iterations of the procedures in each of the three different setups.

\begin{table}[h]
  \centering
  \begin{tabular}{|c|c|c|}
    \hline
Simulation & Fixed parameter & Estimated parameter \\
    \hline
Example~\ref{ex:deuxpoints} & $(a_1,a_2)=(0.4, 0.7)$ & $p^\star=0.3$ \\
    \hline
Example~\ref{ex:deuxpts_3param} & - & $(a_1^\star, a_2^\star, p^\star)=(0.4,0.7,0.3)$\\ 
    \hline
Example~\ref{ex:beta} & -  & $(\alpha^\star, \beta^\star)=(5,1)$ \\
    \hline
  \end{tabular}
  \caption{Parameter values for each experiment.}
  \label{tabl:theta_values}
\end{table}

Figure~\ref{figu:boxplots_cas1_3} shows the  boxplots of our estimator
and \citeauthor{AdEn}'s  estimator obtained  from 1 000  iterations of
the    procedures     in    the    two    Examples~\ref{ex:deuxpoints}
and~\ref{ex:beta}, while Figure~\ref{figu:boxplots_cas2} only displays
these        boxplots        for        our        estimator        in
Example~\ref{ex:deuxpts_3param}.  First, we shall notify that in order to simplify the
visualisation of the results, we removed in the boxplots corresponding
to       Example~\ref{ex:deuxpoints}       (Bottom      panel       of
Figure~\ref{figu:boxplots_cas1_3}) about  0.8\% of outliers  values from
our  estimator, that where  equal to  1.  Indeed  in those  cases, the
likelihood optimisation  procedure did not converge,  resulting in the
arbitrary value $\hat p=1$.  In the same way for Example~\ref{ex:beta}, we removed from the
figure parameter values of \citeauthor{AdEn}'s estimator that were too
large. It corresponds to about 0.7\% of values $\hat \alpha$ larger than 10 (for estimating
$\alpha^\star=5$) and about 0.2\% of values $\hat \beta$ larger than 3 (for estimating
$\beta^\star=1$). In  the following discussion,  we neglect
these rather rare numerical issues. 
We  first  observe  that  the  accuracies  of  the
procedures increase  with the  value of $n$  and thus the  walk length
$T_n$.  We also  note  that  both procedures  are  unbiased. The  main
difference  comes when  considering  the variance  of each  procedure
(related to the width of  the boxplots): our procedure exhibits a much
smaller  variance than \citeauthor{AdEn}'s  one as  well as  a smaller
number of outliers. We stress that \citeauthor{AdEn}'s estimator is expected to exhibit its best performances in Examples~\ref{ex:deuxpoints} and~\ref{ex:beta} that are considered here. Indeed, in these cases, inverting the system of equations that link the parameter to the moments distribution is particularly simple.

\begin{figure}[H]
  \centering
  \includegraphics[height=13cm,width=5cm,angle=-90]{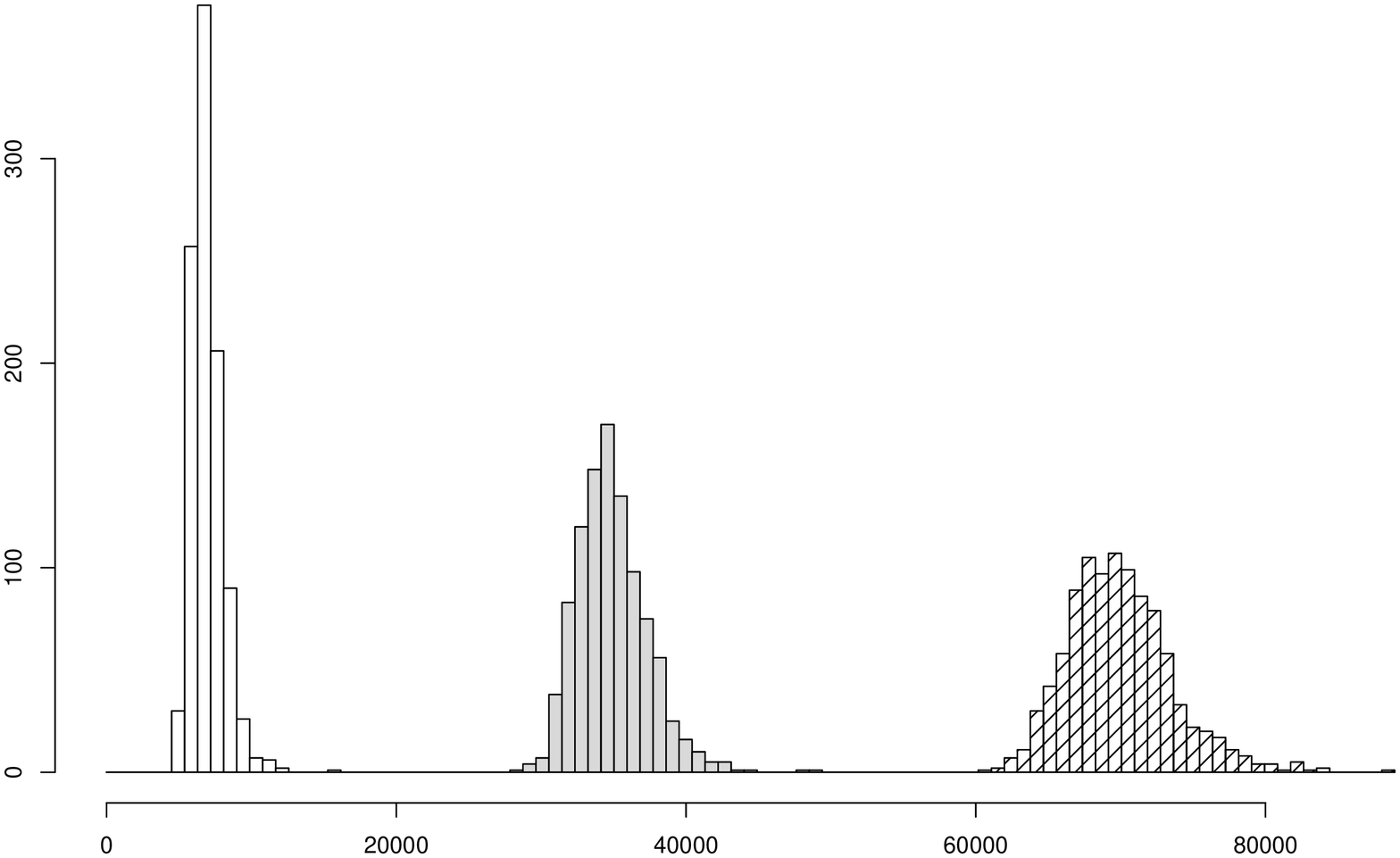} \\
    \includegraphics[height=13cm,width=5cm,angle=-90]{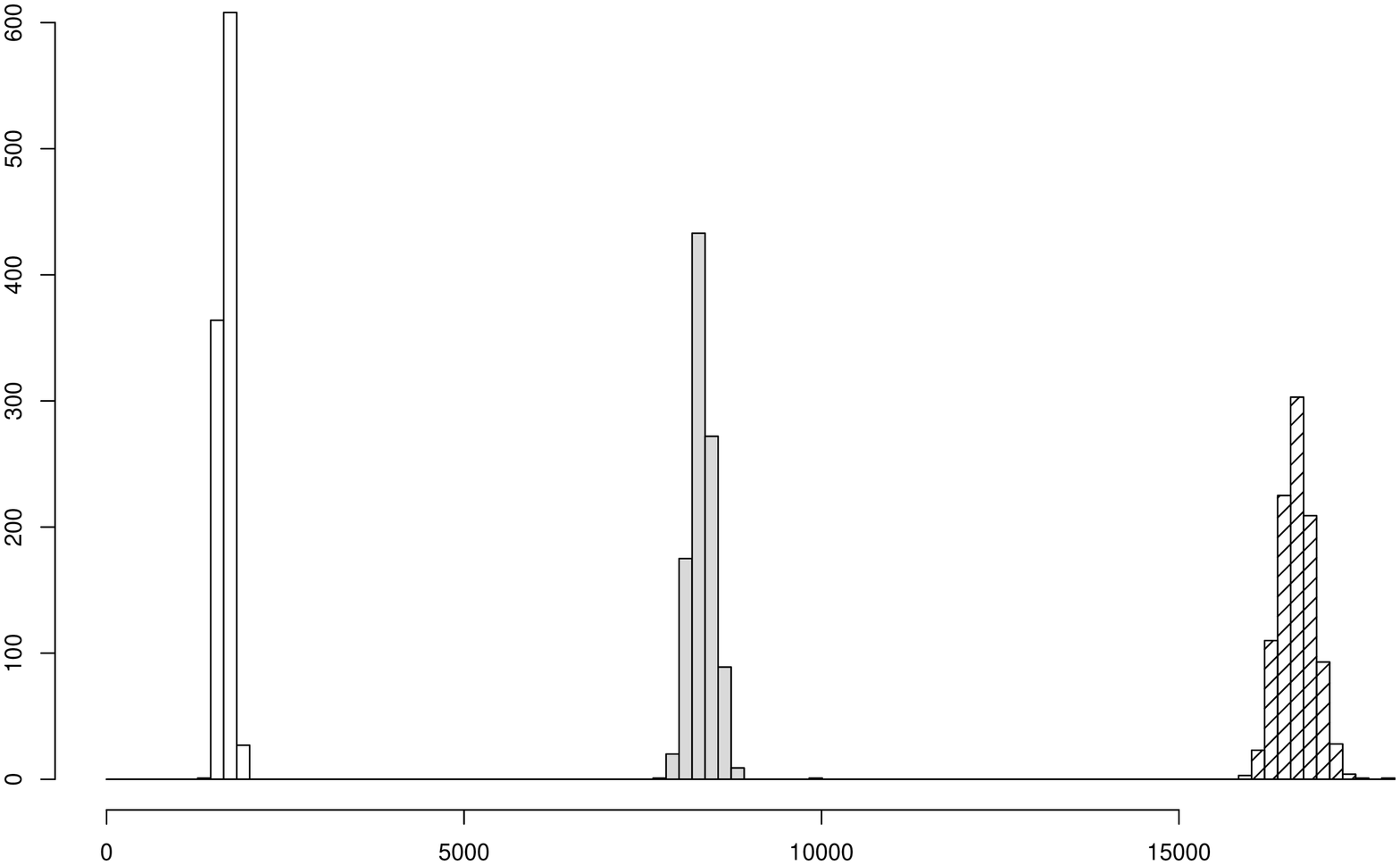}
  \caption{Histograms of  the hitting times $T_n$ obtained  from 1 000
    iterations in each of the two setups and for values $n$ equal to
    1  000 (white),  5 000  (grey) and  10 000  (hatched).  Top panel:
    Example~\ref{ex:deuxpoints}; bottom panel: Example~\ref{ex:beta}.} 
  \label{figu:histos}
\end{figure}

\begin{figure}[H]
  \centering
  \includegraphics[height=13cm,width=5cm,angle=-90]{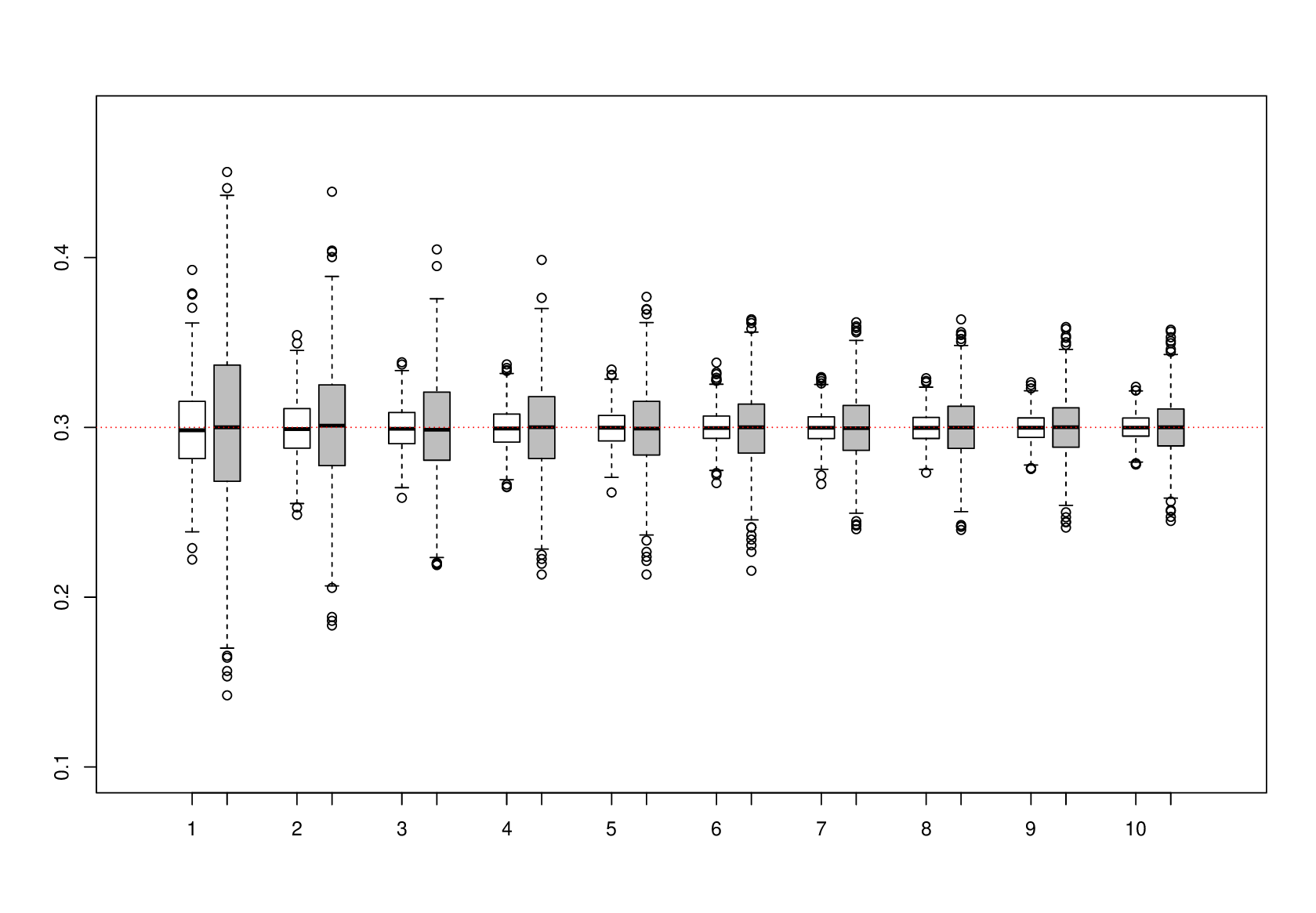} \\
  \includegraphics[height=13cm,width=5cm,angle=-90]{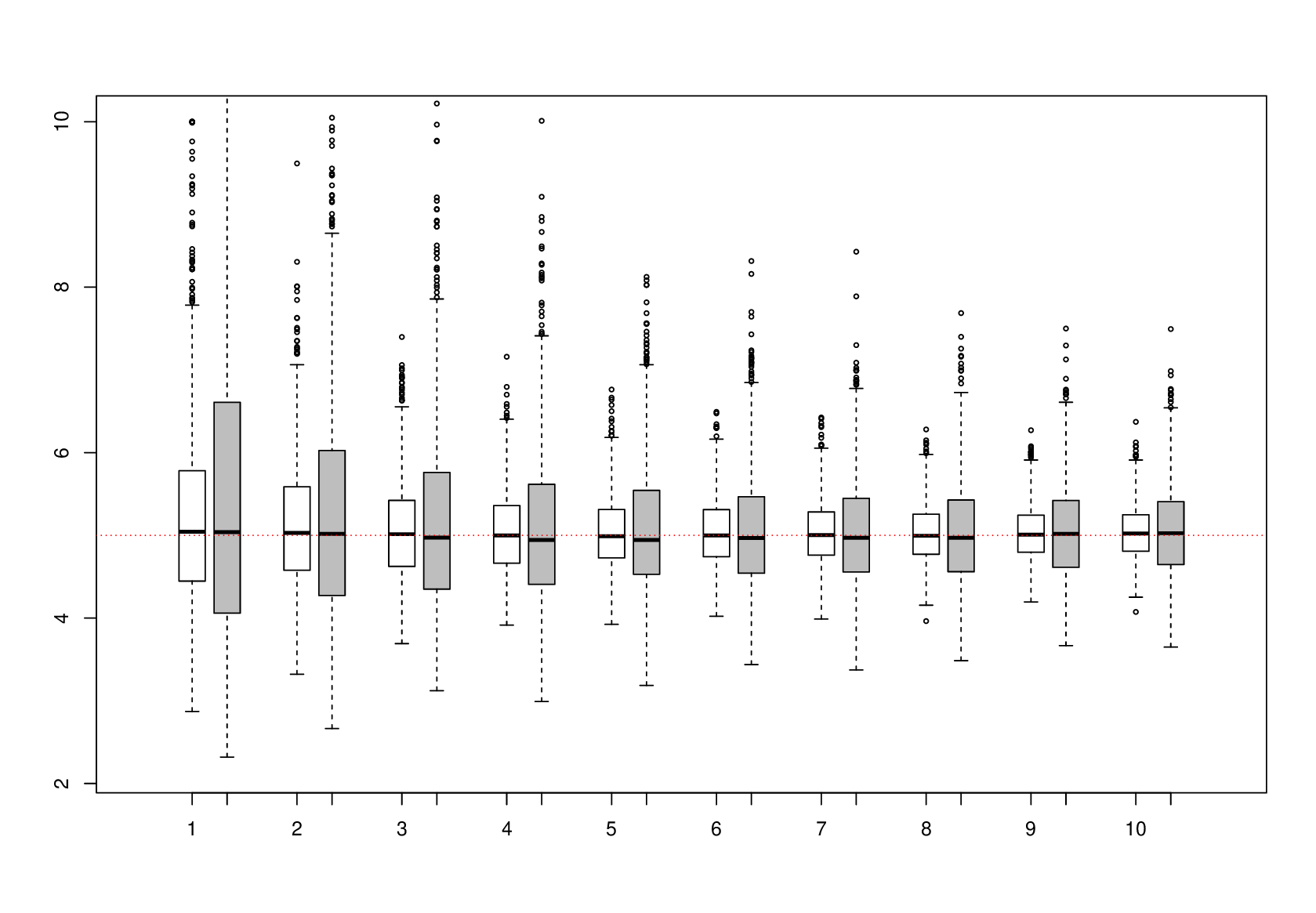}\\
  \includegraphics[height=13cm,width=5cm,angle=-90]{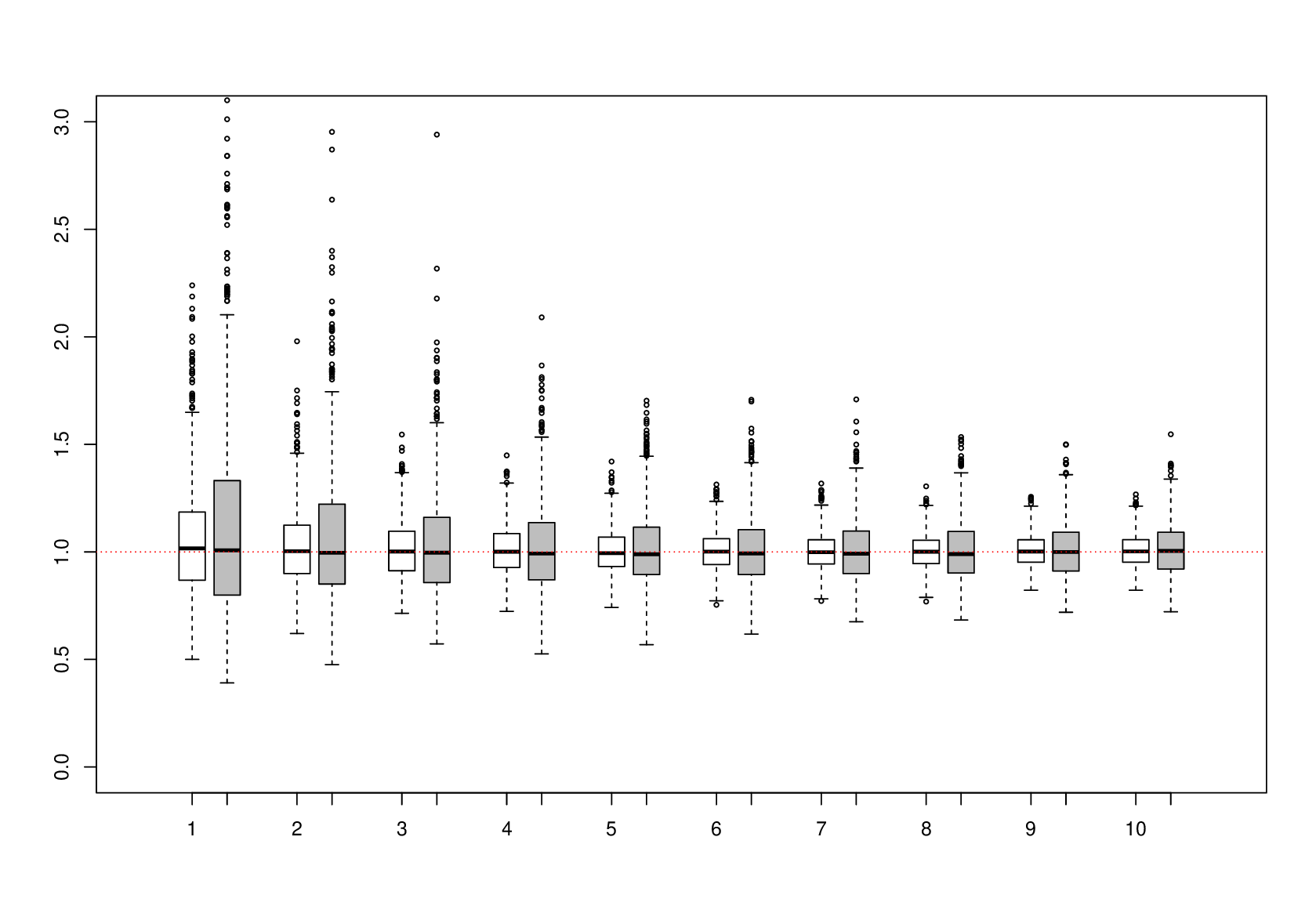}
  \caption{Boxplots   of   our   estimator   (left  and   white)   and
    \citeauthor{AdEn}'s estimator (right and grey) obtained from 1 000
    iterations and for values $n$ ranging
    in $  \{10^3 k ; 1\le k  \le 10 \}$ ($x$-axis  indicates the value
    $k$).    Top   panel   displays   estimation   of   $p^\star$   in
    Example~\ref{ex:deuxpoints}. Second and third panels display  estimation   of
    $\alpha^\star$ (second panel) and $\beta^\star$ (third panel) in Example~\ref{ex:beta}.  The true values    are indicated by horizontal lines.}
  \label{figu:boxplots_cas1_3}
\end{figure}

\begin{figure}[H]
  \centering
  \includegraphics[height=13cm,width=5cm,angle=-90]{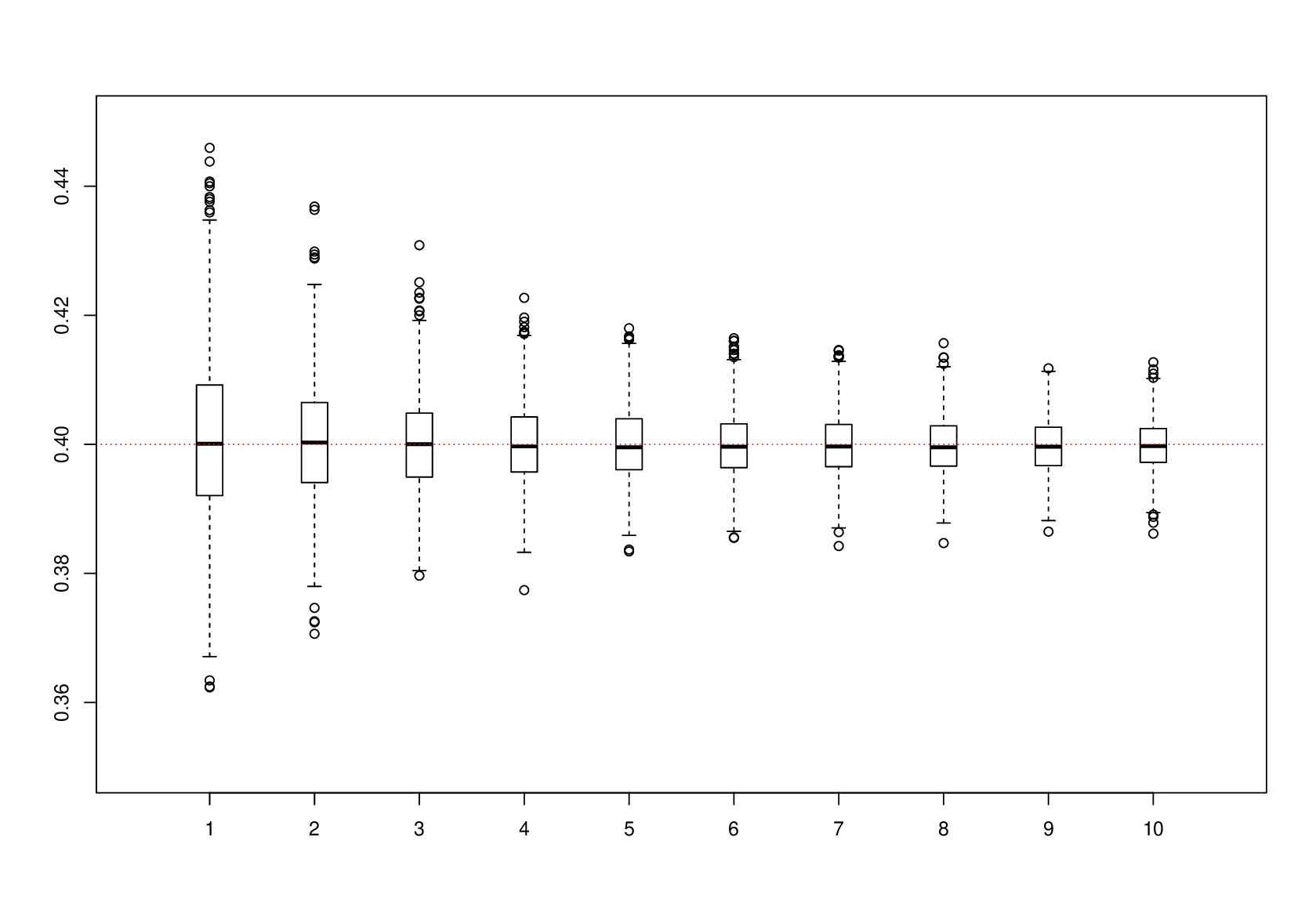} \\
  \includegraphics[height=13cm,width=5cm,angle=-90]{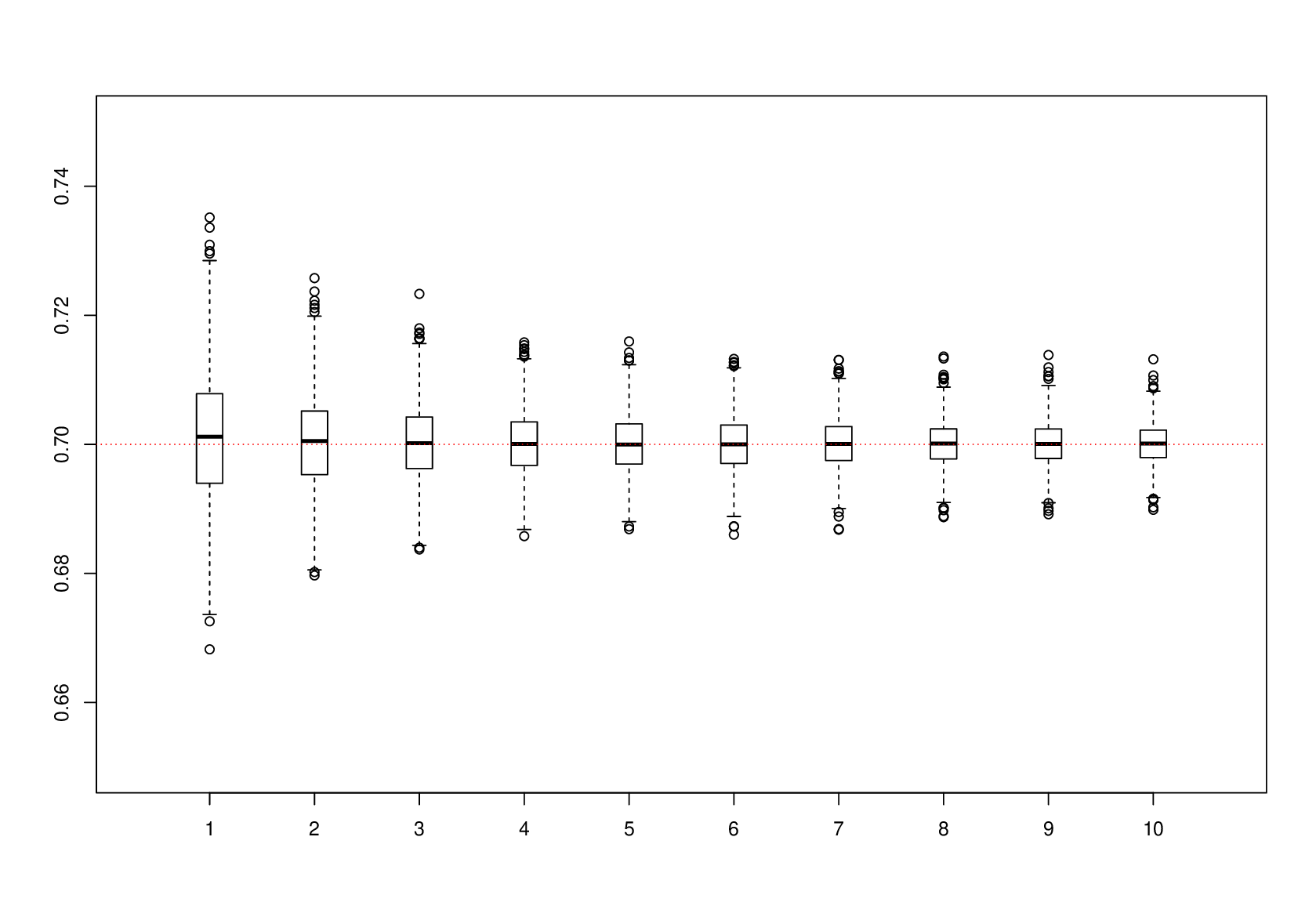} \\
  \includegraphics[height=13cm,width=5cm,angle=-90]{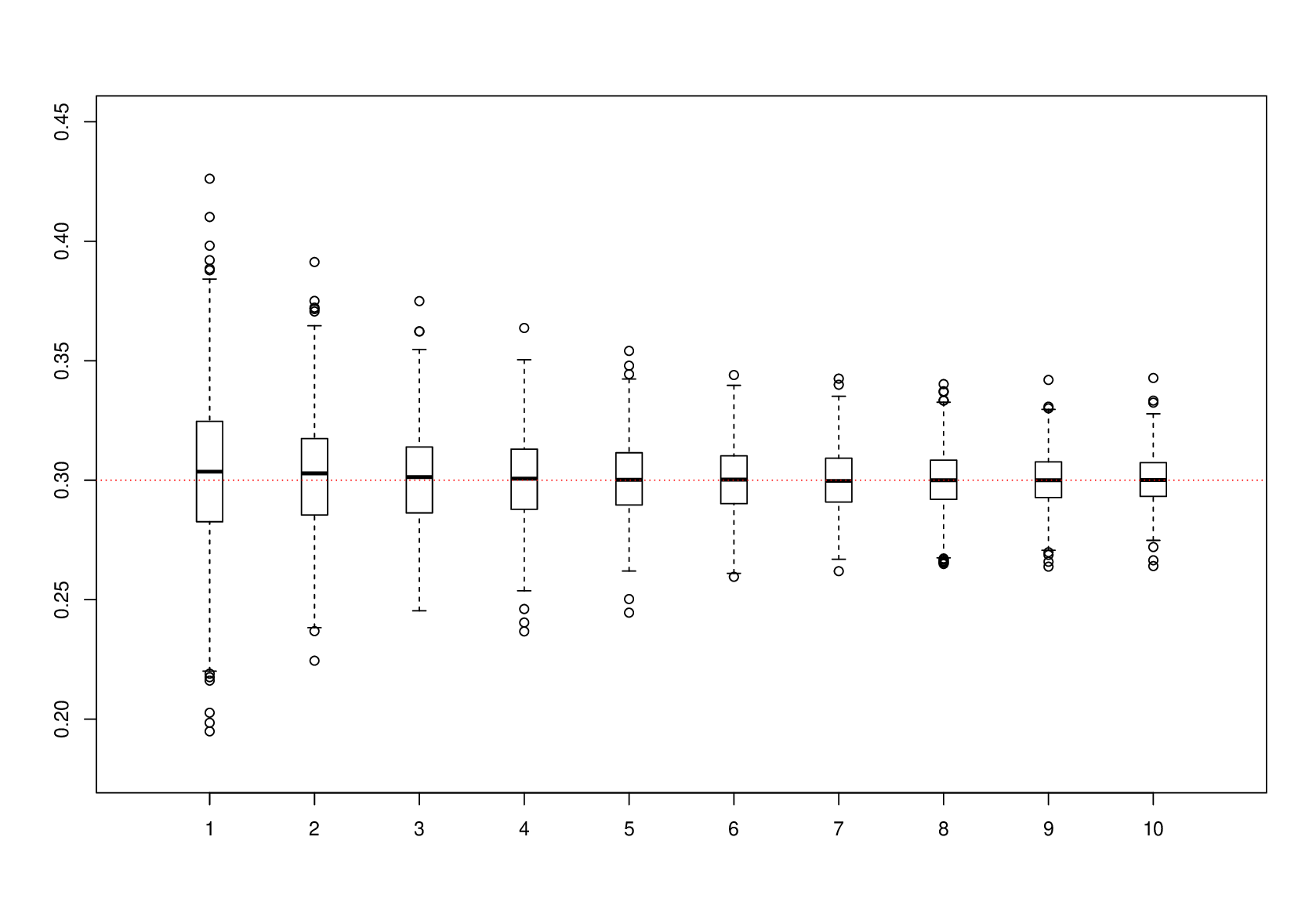} 
  \caption{Boxplots   of   our   estimator   obtained from 1 000
    iterations in   Example~\ref{ex:deuxpts_3param} and for values $n$ ranging
    in $ \{10^3 k ; 1\le k \le 10 \}$ ($x$-axis indicates the value $k$).
Estimation of $a_1^\star$ (top panel), $a_2^\star$ (middle panel) and
   $p^\star$ (bottom panel). The true values
    are indicated by  horizontal lines.}
  \label{figu:boxplots_cas2}
\end{figure}

\bibliographystyle{chicago}
\bibliography{MAMA}

\end{document}